\documentclass[11pt]{article}
\usepackage{xcolor}
\usepackage{geometry}
\usepackage{amsmath, amsthm, amssymb, mathrsfs}
\usepackage{hyperref}
\usepackage[capitalize]{cleveref}
\usepackage{graphicx}

\newtheorem{theorem}{Theorem}

\newtheorem{fact}[theorem]{Fact}

\newtheorem{lemma}[theorem]{Lemma}

\theoremstyle{remark}

\newtheorem{remark}[theorem]{Remark}

\newcommand{\R}{\mathbb{R}}

\newcommand{\eps}{\varepsilon}
\renewcommand{\phi}{\varphi}
\newcommand{\iprod}[2]{\langle #1,#2 \rangle}

\title{An optimal constant for vector balancing with permutations}
\author{Jonathan Niles-Weed, Shay Sadovsky, Jacob Shkrob}
\begin{document}
\maketitle
\begin{abstract}
We present a version of the vector balancing problem in which each vector may be given a sign and a permutation of its coordinates. We prove that this vector balancing problem and its corresponding prefix problem admit an explicit bound, and we further show that it is asymptotically optimal in the dimension. Our method of proof is purely geometric.
\end{abstract}
\section{Introduction}
Dvoretzky's vector balancing problem, posed in 1963  \cite{dvoretzky1963Problem}, is a fundamental question in convex geometry. Given two symmetric convex sets $K$ and $L$ in $\R^n$, the problem asks for a bound on the quantity
\[
\operatorname{vb}(L,K)
:= \sup\left\{
\min_{\varepsilon \in \{-1,1\}^m}
\left\|\sum_{i=1}^m \varepsilon_i u_i\right\|_K
:
m \in \mathbb{N},\ u_1,\ldots,u_m \in L
\right\}.
\]
This problem has been extensively studied in the intervening decades.
Pioneering results due to Gluskin \cite{gluskin1989extremal} and Spencer's \cite{spencer1985six} show that $\text{vb}(B_\infty^n, B_\infty^n)= O(\sqrt{n})$.
The most famous version is the celebrated conjecture of Koml\'os that $\operatorname{vb}(B_2^n,B_\infty^n)= O(1)$. This problem was first studied in the seminal paper of Spencer \cite{spencer1985six}, which gives a bound of $O(\log n)$. Banaszczyk \cite{banaszczyk1998balancing} improved the bound to $O(\sqrt{\log n})$, which stood as the best bound for nearly 30 years before the breakthrough of Bansal and Jiang~\cite{bansal2026decoupling}, which gave a bound of $O((\log n)^{1/4})$ with additional $\log \log n$ factors. Very recently, a preprint claiming a proof of the Koml\'os conjecture (with constant  $3\sqrt{2\pi}$) was posted online~\cite{guo2026vector}. This proof was subsequently simplified and interpreted by Karingula and Lovett~\cite{karingula2026elementaryproofkomlosconjecture}.

The richness of the Koml\'os problem has motivated the study of variants of the vector balancing problem in which the signs $\varepsilon_i$ are replaced by other quantities, for instance by unit-norm vectors~\cite{nikolov_komlos_2013}, jointly Gaussian random variables~\cite{chewi2022gaussian}, or complex numbers~\cite{guillen2026koml}.

The aim of this note is to study a new variant, in which, in addition to a sign choice, the vectors' coordinates may also be permuted, corresponding to a choice of $\varepsilon \in~(\{-1,1\}~\times~S_n)^m$. In this case, we fully solve the discrepancy problem, obtain an explicit dimension-dependent upper bound, and determine the tight asymptotic leading constant.
In fact, we show that it is possible to control all partial sums simultaneously, thereby obtaining a bound on what is known as the prefix discrepancy~\cite{barany1981some,banaszczyk_series_2012,bansal_prefix_2022}.
\begin{theorem}\label{thm:komlos-permutations}
	Let $v_1,\dots, v_m \in B_2^n$, then there exist signs $(\varepsilon_k)_{k=1}^m \in \{-1,1\}^m$ and permutations $(\sigma_k)_{k=1}^m \in S_n$ such that
	\[\max_{1\leq k \leq m}\left \| \sum_{i=1}^k \varepsilon_i \sigma_i(v_i)\right\|_\infty \leq  \frac{\sqrt{n-1}+1}{\sqrt{n}} = 1 + o_n(1) .\]
\end{theorem}
In fact, in our proof, $\varepsilon_k$ and $\sigma_k$ can be chosen ``online,'' that is, one-by-one only on the basis of $v_1, \dots, v_k$~\cite{spencer_balancing_1977}.
In the case when $v_1$ is an elementary basis vector, it is clear that $\left\| \varepsilon_1 \sigma_1(v_1)\right\|_\infty = 1$ for any choice of sign and permutation.
The bound of Theorem~\ref{thm:komlos-permutations} is therefore asymptotically optimal; moreover, an example (see the proof of Theorem~\ref{thm:komlos-permutations}) shows that when $n = 2$, the bound $\tfrac{\sqrt{n-1}+1}{\sqrt{n}} = \sqrt{2}$ is exactly tight.

We remark that this theorem may also be viewed in the setting of colorful vector balancing \cite{ambrus2024colorful}, since choosing a permutation and a sign can be reformulated as choosing, at each step, a single vector from the set $\{\pm \sigma(v_i): \sigma \in S_n\}$.

In Section~\ref{sec:p-q-extension} we extend the proof of Theorem \ref{thm:komlos-permutations} to the case of balancing prefix discrepancy for vectors in the $\ell_p$ unit ball with respect to the $\ell_q$ norm, for any $1 \leq p, q \le \infty$. Using Banaszczyk’s approach along with the property that $B_p^n \subseteq n^{(1/2-1/p)_+}B_2^n$, one can show that $$\mathrm{vb}(B_p^n,B_q^n) \lesssim O_q(n^{1/q + (1/2 - 1/p)_+})$$ for all $p$ and $q$. 
Matching lower bounds show that this dependence on $n$ is optimal \cite{reis2023vector}. In contrast, our second main theorem shows that allowing permutations as in Theorem \ref{thm:komlos-permutations} improves this bound by a factor of $n^{(1/2 - 1/p)_+}$, and that the resulting $O(n^{1/q})$ bound is sharp when $p = \infty$.
Moreover, in the special case when $q = \infty$ and $p < \infty$, we obtain the tight leading constant.

\begin{theorem}\label{thm:komlos-permutations-pq}
	For any $1 \leq p, q \leq \infty$ and any vectors $v_{1},\dots,v_m \in B_{p}^n$, there exist signs $(\varepsilon_{k})_{k=1}^m \in \{-1,1\}^m$ and permutations $(\sigma_k)_{k=1}^m \in S_n$ such that
    \begin{equation}\label{eq: pars}
        \max_{1\leq k\leq m} \left\lVert \sum_{i=1}^k \varepsilon_i \sigma_i(v_i) \right\rVert_q = O(n^{1/q}).
    \end{equation}
    In particular, when $q = \infty$, these signs and permutations satisfy
    \begin{equation}\label{eq:pq_sharp}
        \max_{1\leq k\leq m} \left \lVert \sum_{i=1}^k \varepsilon_i \sigma_{i}(v_i) \right \rVert_\infty \leq \begin{cases}
           1 + o_n(1) & 1 \leq p < \infty \\
           3 + o_n(1) & p = \infty
        \end{cases}.
    \end{equation}
\end{theorem}
\begin{remark}
	The bound in~\eqref{eq: pars} is unimprovable: when $p = \infty$, the single vector $v_1 = \mathbf{1}$ satisfies $\lVert \varepsilon_1 \sigma_1(v_1) \rVert_q = n^{1/q}$ for every choice of sign and permutation.
	Similarly, choosing $v_1$ to be an elementary basis vector shows that the bound in~\eqref{eq:pq_sharp} for $1 \leq p < \infty$ is asymptotically sharp.
    Perhaps surprisingly, the separation between the $p < \infty$ and $p = \infty$ cases is not an artifact of our proof. Indeed, the example $v_1 = (1,1,\dots,1)^\top$ and $v_2 = (1,-1,0,\dots,0)^{\top}$ for $n\geq 2$ shows that the constant is at least $2$. Obtaining the sharp constant for $p = q= \infty$ is an open problem.
\end{remark}

Our main tool in the proof of Theorem~\ref{thm:komlos-permutations} is a tight bound for vectors in the subspace  $\mathbf{1}^\perp$ orthogonal to the vector $\mathbf{1} = (1,1,\dots, 1)$.
In fact, vectors in this subspace can be balanced by permutations alone.
\begin{lemma}\label{lem:tight-zero-sum}
	Let $v_1,\dots, v_m \in \mathbf{1}^\perp \cap B_2^n$, then there exist permutations $(\sigma_k)_{k=1}^m \in S_n$ such that
	\[\max_{1\leq k \leq m} \left \| \sum_{i=1}^k \sigma_i(v_i)\right\|_\infty \le \sqrt{\frac{n-1}{n}}.\]
\end{lemma}
The bound of Lemma~\ref{lem:tight-zero-sum} is tight for all $n$.
We postpone the proof to the next section.
Our proof is constructive: we show that the permutations can be chosen greedily so that the partial sums remain in the regular permutahedron.

With the lemma in hand, the proof of the main theorem is nearly immediate.
\begin{proof}[Proof of Theorem \ref{thm:komlos-permutations}]
	For each vector from the list, denote $v_k = u_k+w_k$, where $u_k = P_{\mathbf{1}^\perp} v_k$ and $w_k = P_{\mathbf{1}} v_k$. The vectors $w_1,\dots,w_k$ are parallel to $\mathbf{1}$, and as they are all of length at most 1, each one has sup-norm at most $\frac{1}{\sqrt{n}}$.
	Write $w_k = t_k \mathbf{1}$, so that $|t_k| \leq \frac{1}{\sqrt{n}}$.
	We choose the signs greedily: let $\eps_k \in \{-1,1\}$ be such that $\eps_k t_k \sum_{i=1}^{k-1} \eps_i t_i \leq 0$.
	Since $|s + t| \leq \max(|s|, |t|)$ whenever $st \leq 0$, induction on $k$ gives
	\[\max_{1 \leq k \leq m} \left\| \sum_{i=1}^k \eps_i w_i \right\|_\infty = \max_{1 \leq k \leq m} \left| \sum_{i=1}^k \eps_i t_i \right| \le \frac{1}{\sqrt{n}}.\]
	Since the vectors $\eps_k u_k$ all lie in $B_2^n \cap \mathbf{1}^\perp$, Lemma~\ref{lem:tight-zero-sum} yields permutations $\sigma_k$ such that
	\[\max_{1 \leq k \leq m} \left \| \sum_{i=1}^k \sigma_i(\eps_i u_i)\right\|_\infty \le \sqrt{\frac{n-1}{n}}.\]

	Finally, for any vector $w$ parallel to $\mathbf{1}$ and any permutation $\sigma$, $\sigma(w)=w$. So for each $1 \leq k \leq m$, the triangle inequality gives
	\[ \left \| \sum_{i=1}^k \eps_i \sigma_i(v_i)\right\|_\infty \le \left \| \sum_{i=1}^k \sigma_i(\eps_i u_i)\right\|_\infty + \left \| \sum_{i=1}^k \eps_i w_i\right\|_\infty \le \frac{\sqrt{n-1}+1}{\sqrt{n}}.\]
	Each $\eps_k$ and $\sigma_k$ depends only on $v_1, \dots, v_k$, so this construction is online.
    The upper bound is attained in dimension 2 by taking \[v_1 = \left(\frac{1}{\sqrt{2}},-\frac{1}{\sqrt{2}}\right) , \quad  v_2 = \left(\frac{1}{\sqrt{2}},\frac{1}{\sqrt{2}}\right). \]
\end{proof}
 
\section{The tight bound in \texorpdfstring{$\mathbf{1}^\perp$}{the orthogonal complement of 1}}
The goal of this section is to prove Lemma~\ref{lem:tight-zero-sum}.
To find a tight bound, we recall the definition of majorization: A vector $x\in \R^n$ is majorized by $y\in \R^n$, if, for their decreasing rearrangements $x^*, y^*$
\[ \forall 1\le k \le n \ \sum_{i=1}^k x^*_i \le \sum_{i=1}^k y^*_i \ , \quad \sum_{i=1}^n x_i = \sum_{i=1}^n y_i .\]
A theorem of Rado \cite{rado1952inequality} states that $x$ is majorized by $y$ if and only if $x$ is in the convex hull of all $n!$ permutations of the coordinates of $y$.
Hence, we may conclude that $x~\in~\text{conv}(\sigma(y):~\sigma \in S_n ).$ 
(Such a polytope is called a permutahedron. See \cite{postnikov2009permutohedra} for a thorough introduction.)

The above discussion suggests the following strategy. We will construct a permutahedron $P$ containing $B_2^n\cap \mathbf{1}^\perp$ satisfying the following closure property: for any $x, y \in P$, there exists $\sigma \in S_n$ such that $x + \sigma(y) \in P$.
Given such a set and vectors $v_1, \dots, v_m$ in $B_2^n\cap \mathbf{1}^\perp$, we can use the closure property to ensure that the partial sums $\sum_{i=1}^k \sigma_i(v_i)$ remain in $P$ for all $k$.
Bounding the $\ell_\infty$ diameter of $P$ will then suffice to prove the claim.

To this end, define the centered regular permutahedron (sometimes called the Coxeter permutahedron) by
\begin{equation}\label{def:permutahedron}
    P:= \text{conv}(\sigma(w):\sigma \in S_n ),\ \ w=(n-1, n-3,\dots, 3-n, 1-n).
\end{equation}
We claim that it is closed under the operation of taking sums of increasing and decreasing rearrangements.

\begin{lemma}\label{lem:permutahedron}
	Denote $H_+=\{x \in \R^n: x_1\ge x_2 \ge...\ge x_n\}$, similarly $H_-=\{x \in \R^n: x_1\le x_2 \le...\le x_n\}$. It holds that
    \begin{equation} \label{eq:permutahedron-inclusion}
        (P\cap H_+)+(P\cap H_-) \subseteq P,
    \end{equation}where $+$ is the Minkowski sum.
\end{lemma}

We will use the following facts about vectors in $P$.
\begin{fact}
	Let $P\subset \R^n$ be the centered regular permutahedron. Then, $P$ is a zonotope given by
	\[\sum_{1\le i< j \le n}[-(e_i - e_j), e_i-e_j].\]
	In particular, it is the image of the cube $[-1,1]^N$, where $N= \binom{n}{2}$ under the projection
	$L:\R^N\to \R^n$ given by
	\begin{equation}\label{eq:proj-to-premu}
			L(v)=\sum_{(i,j)\in {n\choose 2}} v_{(i,j)}(e_i-e_j).
	\end{equation}

\end{fact}
This is sometimes called the graphical zonotope. For a proof, see \cite{postnikov2009permutohedra} and see \cite{ardila2020arithmetic} for more on the specific case of the centered regular permutahedron. For a short introduction to the permutahedron, see \cite{barvinok2025course}, and for the classical reference see \cite{marshall1979inequalities}.

\begin{proof}[Proof of Lemma \ref{lem:permutahedron}]
	The main step in the proof is to show that, for $N= {n\choose 2}$ and $L$ as in \eqref{eq:proj-to-premu},

	\begin{equation}\label{eq:inclusion-decreasing}
	P\cap H_+ \subseteq \sum_{i<j} [0, e_i-e_j] = L([0,1]^N).
\end{equation}
The lemma would then follow directly, since $H_+=-H_-$ and $P=-P$,
\begin{align*}
	(P\cap H_+)+(P\cap H_-) &= (P\cap H_+)-(P\cap H_+)  \\
	& \subseteq  \sum_{i<j} [0, e_i-e_j] + \sum_{i<j} [e_j-e_i,0] \\
	&= L([0,1]^N) + L([-1,0]^N) = L([-1,1]^N)=P. 
\end{align*}

Let us proceed with proving \eqref{eq:inclusion-decreasing}. The coordinates of $P\cap H_+$ are decreasing, and are majorized by $w$, so this set is defined by the following inequalities in the $(n-1)$ dimensional space $\mathbf{1}^\perp = \{x: \sum_{i=1}^n x_i = 0 \}$:
\begin{align}
        &\sum_{i=1}^k x_i \le \sum_{i=1}^k w_i = k(n-k) & 1\le k \le n-1 \label{eq:majorization-cond}\\
        &x_k \ge x_{k+1} & 1\le k \le n-1. \label{eq:monotonicity-cond}
\end{align}
A vertex $x$ of the polytope must attain equality in $n-1$ of the $2(n-1)$ inequalities above. If the $k$-th coordinate has equality in $\eqref{eq:majorization-cond}$, then
\[x_k \ge k(n-k) - (k-1)(n-k+1) \text{ and } x_{k+1}\le (k+1)(n-k-1) - k(n-k)\]
which imply that the equality in \eqref{eq:monotonicity-cond} cannot hold. Since $x_n = -\sum_{i=1}^{n-1} x_i$, we conclude that any extremal point of $P\cap H_+$ is by equality in either \eqref{eq:majorization-cond} or \eqref{eq:monotonicity-cond} in each of the coordinates $x_1,\dots, x_{n-1}$.

Let us assume that the following coordinates have equality in \eqref{eq:majorization-cond}
\[0<k_1<k_2<\dots<k_m< n,\]
then, with the conventions $k_0 := 0$ and $k_{m+1} := n$, for each $0 \le r \le m$ and each $k_r~<~k~<~k_{r+1}$ equality holds in \eqref{eq:monotonicity-cond}, i.e. there exists some constant $c_r$ such that
\[c_r=x_{k_r +1}=x_{k_r +2} = \dots = x_{k_{r+1}}.\] 
Since there are $k_{r+1}-k_r$ equal numbers above, and due to the equality \eqref{eq:majorization-cond} at the endpoints,
\[c_r (k_{r+1}-k_r) = \sum_{i=k_r +1}^{k_{r+1}} x_i = \sum_{i=1}^{k_{r+1}}x_i-\sum_{i=1}^{k_r}x_i = k_{r+1}(n-k_{r+1})-k_{r}(n-k_{r})\]
so $c_r = \frac{k_{r+1}(n-k_{r+1})-k_{r}(n-k_{r})}{(k_{r+1}-k_r)} = n-k_r - k_{r+1}$.

We now produce points $u\in [0,1]^N$ that project to the vertex, i.e. $L(u)=x$, which will complete the proof. Denote
\[S_t=\{k: k_t < k \le k_{t+1}\}, \quad 0 \le t \le m,\]
and define
\[u_{ij}=\begin{cases}
	0 & \exists t \text{ s.t. } j,i\in S_t\\
	1 & \text{otherwise}
\end{cases},\]
namely that the vector indexed at $(i,j)$ is 1 whenever $i$ and $j$ are in different index sets $S_t,S_s$.

We now verify that $L(u)_k = x_k$, where $k\in S_r$ for a fixed $r$, recalling that we have shown in this case that $x_k=c_r$. We must count exactly how many times $e_k$ appears in 
\[\sum_{i<j,\ i\in S_t \ j\in S_s} e_i-e_j\]
when $s\neq t$. The $k$-th coordinate $e_k$ plays the role of the first vector $e_i$ when $j\in \cup_{t=r+1}^m S_t$, so $n-k_{r+1}$ times. It plays the role of the second vector $e_j$ when $i\in \cup_{t=0}^{r-1} S_t$, exactly $k_r$ times. Since in the latter case it appears with a negative sign, we conclude $L(u)_k =n - k_{r+1}-k_r$, exactly matching $c_r$, and the proof is complete.
\end{proof}

Applying the strategy proposed above requires scaling $P$ so that it contains $B_2^n \cap \mathbf{1}^\perp$.
\begin{fact}\label{fact:smallest-permutahedron}
	Let $  P\subseteq \R^n$ be the centered regular permutahedron. If $C_n = \sqrt{\frac{1}{n(n-1)}} =\frac{1}{n-1}   \sqrt{\frac{n-1 }{n}}$, then $B_2^n \cap \mathbf{1}^\perp \subseteq C_n  P$.
\end{fact}
\begin{proof}
	Let $x\in B_2^n\cap \mathbf{1}^\perp$, we must check that it is majorized by $C_n  w =\frac{1}{n-1}   \sqrt{\frac{n-1 }{n}} (n-1, n-3,\dots, 3-n, 1-n)$. Indeed, for any index set $S$ with $|S|=k$ ($0<k<n$), 
	\[ \sum_{i\in S} x_i = \iprod{x }{1_S - \frac{k}{n}\mathbf{1}}\le \|x\|_2 \sqrt{\frac{k(n-k)}{n}} \le \sqrt{\frac{n-1}{n}} \sum_{i=1}^{k}\frac{w_i}{n-1}\]
	where in the first equality we used the fact that $x \in \mathbf{1}^\perp$.
	\end{proof}
\begin{remark}
	It is evident that there is equality when $k=1$, and it is simple to check that the vector which attains the equality is $(\sqrt{\frac{n-1}{n}},-\frac{1}{\sqrt{n(n-1)}},\dots,  -\frac{1}{\sqrt{n(n-1)}})$.
\end{remark}
We can now prove a tight bound in the main Lemma \ref{lem:tight-zero-sum}.
\begin{proof}[Proof of Lemma \ref{lem:tight-zero-sum}]
	According to Fact \ref{fact:smallest-permutahedron}, $v_1,\dots,v_m \in C_n P$. We prove Lemma \ref{lem:tight-zero-sum} by an induction on the prefix $k \in [m]$, ensuring the same permutations are used at each inductive step. For $k=1$, the statement is trivial. For general $k$,
	assume we have constructed permutations $\sigma_{1},\dots,\sigma_{k-1}$ such that 
	\[u = \sum_{i=1}^{k-1} \sigma_i(v_i) \in C_n P.\]
    Take $\sigma_+$ to be the decreasing rearrangement of $u$ and $\sigma_-$ to be the increasing rearrangement of $v_k$. Again by Lemma \ref{lem:permutahedron},
    \[\sigma_+(u)+\sigma_-(v_k)\in (C_n P\cap H_+)+(C_n P\cap H_- ) \subseteq C_n P.\]
    Set $\sigma_{k} = \sigma_+^{-1} \circ \sigma_{-}$. Then
    $$
        \sum_{i=1}^k \sigma_{i}(v_i) = \sigma_+^{-1} (\sigma_+(u)+\sigma_-(v_k))\in C_n P.
    $$
    We obtain by induction that there exist permutations such that $\sum_{i=1}^k \sigma_i(v_i) \in C_n P$ for all $k \in [m]$.
    
    The bound then follows from the fact that $C_n P \subseteq \sqrt{\frac{n-1}{n}}[-1,1]^n$, since $C_n P$ is generated by the vector $\sqrt{\frac{n-1}{n}} \left(1, \frac{n-3}{n-1},\dots, -\frac{n-3}{n-1},-1\right)$.
\end{proof}

\section{Extensions for \texorpdfstring{$\ell_p$-to-$\ell_q$}{lp-to-lq} permutation vector balancing}\label{sec:p-q-extension}
In this section, we prove Theorem~\ref{thm:komlos-permutations-pq}, which extends our main result to the case of $\ell_p$-to-$\ell_q$ permutation vector balancing. Prior work on vector discrepancy \cite{reis2023vector} shows that for $m\leq n$, $2\leq p \leq q \leq \infty$, and $\{p,q\} \neq \{2,\infty\},$
$$
    \min_{\eps \in \left\{\pm 1\right\}^m}\left\lVert\sum_{i=1}^m \eps_i v_i\right\rVert_q \lesssim \frac{\sqrt{\min(p, \log(\tfrac{2n}{m}))}}{1/2 - 1/p + 1/q}m^{1/2 + 1/q - 1/p}.
$$

Our proof follows the same steps as in the proof of Theorem~\ref{thm:komlos-permutations}: we greedily choose signs to balance the component in the $\mathbf 1$ direction, and use permutations to ensure that the partial sums lie in a suitable rescaling of the permutahedron $P$ that contains $P_{ \mathbf{1}^\perp}B_p^n$. 

Let $R_{n,q} = \max_{x \in P}\lVert x\rVert_{q}$ and $C_{n,p} > 0$ be the quantity
$$
    C_{n,p} = \begin{cases}
        \frac{1}{n}\left( 1 + (n-1)^{1-p^\ast} \right)^{1/p^\star} & 1 < p < \infty \\
        \frac{1}{n} & p = 1 \\
        \frac{2}{n} & p = \infty
    \end{cases},
$$
where $p^\ast = \frac{p}{p-1}$.
We will show that $P_{ \mathbf{1}^\perp}B_p^n \subseteq C_{n,p} P$.
As a sanity check, for the case $p = 2$, we exactly recover the constant $C_{n} = C_{n,2} = \sqrt{\frac{1}{n(n-1)}}$ obtained in Fact \ref{fact:smallest-permutahedron}. 

We prove the following explicit version of Theorem \ref{thm:komlos-permutations-pq}.
\begin{theorem}
    If $1\leq q \leq \infty$, for any vectors $v_{1},\dots,v_m \in B_{p}^n$ for $1\leq p \leq \infty$, there exist signs $(\varepsilon_{k})_{k=1}^m \in \{-1,1\}^m$ and permutations $(\sigma_k)_{k=1}^m \in S_n$ such that
    \begin{equation}\label{eq:pars-explicit}
        \max_{1\leq k\leq m} \left\lVert \sum_{i=1}^k \varepsilon_i \sigma_i(v_i) \right\rVert_q \leq R_{n,q}C_{n,p} + n^{1/q - 1/p}=O(n^{1/q}).
    \end{equation}
    Moreover, $R_{n, \infty} C_{n, p} = 1 + o_n(1)$ when $p < \infty$ and $R_{n, \infty} C_{n, \infty} = 2 + o_n(1)$.
\end{theorem}
This Theorem directly follows from the following analogue of Lemma~\ref{lem:tight-zero-sum}.
\begin{lemma}\label{lem: l_q - l_p}
    Fix $1 \leq p \leq \infty$ and $1 \leq q \leq \infty$. 
    Let $R_{n, q}$ and $C_{n, p}$ be as above.
    Then for any vectors $v_1,\dots, v_m \in P_{\mathbf{1}^\perp}B_{p}^n$ each $i=1,\dots,m$, there exist permutations $(\sigma_k)_{k=1}^m \in S_n$ such that
    $$
        \max_{1\leq k \leq m}\left\lVert \sum_{i=1}^k \sigma_i(v_i) \right\rVert_q \leq R_{n,q} C_{n,p}.
    $$
\end{lemma}

\begin{proof}
   Following the proof of Lemma~\ref{lem:tight-zero-sum}, it suffices to show that $P_{\mathbf{1}^\perp}B_{p}^n \subseteq C_{n,p} P$ and $P \subseteq R_{n,q}B_{q}^n$. This would imply that the same rearrangement procedure as in Lemma \ref{lem:tight-zero-sum} will apply verbatim. Since the maximum $\ell_q$ norm of elements of $P$ occurs at its vertices,
    $$R_{n,q} = \begin{cases}(\sum_{k=1}^{n} |n - 2k + 1|^q)^{1/q} \sim (1+q)^{-1/q}n^{1+1/q} & q < \infty\\
    n-1 & q = \infty
    \end{cases}.
    $$
    Suppose $1 < p < \infty$. To find $C_{n,p}$, for index set $S \subsetneq [n]$ of size $|S| = k$ and $x = P_{\mathbf{1}^\perp}v \in P_{\mathbf{1}^\perp}B_p^n,$ if $p^\ast$ is the H\"older conjugate of $p$,
    $$ 
        \sum_{i \in S}x_i \leq \lVert v \rVert_p \lVert{1}_S - \frac{k}{n}\mathbf{1}\rVert_{p^\ast}.
    $$
    Hence, optimizing over all such index sets gives
    $$
        \max_{S\subseteq [n]: |S| = k}\sum_{i \in S} x_i \leq d_{n,p}(k), \quad d_{n,p}(k) = \left(\left(\frac{k}{n}\right)^{p^\ast}(n-k) + \left(1-\frac{k}{n}\right)^{p^\ast} k\right)^{1/p^\ast}.
    $$
    Therefore by majorization, the optimal constant is
    $$
        C_{n,p} = \max_{1\leq k \leq n-1} \frac{d_{n,p}(k)}{k(n-k)} = \max_{1\leq k \leq n-1} \frac{1}{n}\left(k^{1-p^\ast} + (n-k)^{1-p^{\ast}}\right)^{1/p^{\ast}}.
    $$
    Since $x \mapsto (n-x)^{1-p^\ast}+x^{1-p^\ast}$ is convex in $[1,n-1]$, the maximum occurs at the boundary, giving the final constant
    $$
        C_{n,p} = \frac{1}{n}\left(1 + {(n-1)^{1-p^\ast}}\right)^{1/p^\ast} = \frac{1}{n(n-1)}\left((n-1)^{p^\ast} + (n-1)\right)^{1/p^\ast}.
    $$
    For the case $p = 1$, the same proof gives $d_{n,1}(k) = \lVert 1_S - \frac{k}{n}\mathbf{1}\rVert_\infty=\frac{\max(n-k,k)}{n}$ and so
    $$
        C_{n,1} = \max_{1\leq k \leq n-1}\frac{d_{n,1}(k)}{k(n-k)} = \frac{1}{n}.
    $$
    Similar computations for the $p=\infty$ case gives $C_{n,\infty} = \frac{2}{n}$.
    
\end{proof}
\begin{proof}[Proof of Theorem \ref{thm:komlos-permutations-pq}]
Write $v_i = u_i + w_i$, where $u_i = P_{\mathbf{1}^\perp}v_i$ and $ w_i = P_{\mathbf{1}}v_i$. Every vector in $P_{\mathbf{1}}B_p^n$ has entries of absolute value at most $\tfrac{1}{n^{1/p}}$, so $w_i = t_i \mathbf{1}$ with $|t_i| \leq n^{-1/p}$.
Choosing the signs greedily as in the proof of Theorem~\ref{thm:komlos-permutations}, so that $\eps_k t_k  \sum_{i=1}^{k-1} \eps_i t_i \leq 0$, gives $\left|\sum_{i=1}^k \eps_i t_i\right|~\leq~n^{-1/p}$ for all $k$, and hence
$$
    \max_{1 \leq k \leq m} \left\lVert \sum_{i=1}^k \eps_i w_i \right\rVert_q \leq n^{1/q - 1/p}.
$$
The vectors $\eps_i u_i$ lie in $P_{\mathbf{1}^\perp}B_p^n$, so Lemma~\ref{lem: l_q - l_p} yields permutations $\sigma_i$ such that
$$
    \max_{1\leq k \leq m}\left\lVert \sum_{i=1}^k \sigma_i(\eps_i u_i) \right\rVert_q \leq R_{n,q} C_{n,p}.
$$
Since $\sigma_i(w_i) = w_i$, the triangle inequality gives, for each $1 \leq k \leq m$,
$$
    \left\lVert \sum_{i=1}^k \eps_i \sigma_i(v_i) \right\rVert_q \leq \left\lVert \sum_{i=1}^k \sigma_i(\eps_i u_i) \right\rVert_q + \left\lVert \sum_{i=1}^k \eps_i w_i \right\rVert_q \leq R_{n,q} C_{n,p} + n^{1/q - 1/p}.
$$
The asymptotic claims follow since $R_{n,\infty} = n-1$, so that $$R_{n,\infty}C_{n,p} = \frac{1}{n}\left((n-1)^{p^\ast} + (n-1)\right)^{1/p^\ast} \to 1$$ for $1 < p < \infty$, while $R_{n,\infty}C_{n,1} = \frac{n-1}{n}$ and $R_{n,\infty}C_{n,\infty} = \frac{2(n-1)}{n}$.
\end{proof}

\begin{remark}
    The method presented in this section may be applied to permutation balancing in any two norms. Namely, given $K,L \subset \mathbb{R}^n$ origin-symmetric and vectors $v_{1},\dots,v_m \in L$ there exist signs $(\varepsilon_{k})_{k=1}^m \in \{-1,1\}^m$ and permutations $(\sigma_k)_{k=1}^m \in S_n$ such that
    \[\max_{1\leq k \leq m}\left \| \sum_{i=1}^k \varepsilon_i \sigma_i(v_i)\right\|_K \le \|w\|_K \max_{x \in P_{\mathbf{1}^\perp} L} \|x\|_P + \frac{\|\mathbf 1\|_K}{\|\mathbf 1\|_L},\]
    with $w$ as defined in \eqref{def:permutahedron}.
\end{remark}
\subsection*{Acknowledgements}
JNW is supported by National Science Foundation grant DMS-2339829.

An LLM was used for exploration of the behavior of the problem of balancing with permutations, calculating examples that helped the authors to conclude that the permutahedron is the correct object of study. The LLM was later used for assistance with the proof of Lemma \ref{lem:permutahedron} and for editing this manuscript and checking for errors. The authors are fully responsible for the mathematical integrity of this manuscript and its correctness.

\bibliographystyle{amsplain}
\addcontentsline{toc}{section}{\refname}\bibliography{permutations}

\bigskip
\noindent\begin{minipage}{\textwidth}
\small
\textsc{Jonathan Niles-Weed}\\
Courant Institute of Mathematical Sciences and Center for Data Science,\\
New York University\\
\textit{Email:} \href{mailto:jnw@cims.nyu.edu}{\nolinkurl{jnw@cims.nyu.edu}}
\end{minipage}

\medskip
\noindent\begin{minipage}{\textwidth}
\small
\textsc{Shay Sadovsky}\\
Courant Institute of Mathematical Sciences, New York University\\
\textit{Email:} \href{mailto:ss20011@nyu.edu}{\nolinkurl{ss20011@nyu.edu}}
\end{minipage}

\medskip
\noindent\begin{minipage}{\textwidth}
\small
\textsc{Jacob Shkrob}\\
Courant Institute of Mathematical Sciences, New York University\\
\textit{Email:} \href{mailto:jas10184@nyu.edu}{\nolinkurl{jas10184@nyu.edu}}
\end{minipage}
\end{document}